\documentclass[11pt]{amsart}
\usepackage[english]{babel}
\usepackage{multirow}
\usepackage{colortbl}
\usepackage[utf8]{inputenc}
\usepackage{amsmath, amsthm,  amsfonts, amssymb, mathrsfs,textcomp}
\usepackage{geometry}
\usepackage{tikz}
\usepackage{subcaption}
\usepackage{graphicx}
\usepackage[makeroom]{cancel}

\usetikzlibrary{patterns}  	
\usetikzlibrary{intersections}  
\usetikzlibrary{shapes}  	
\usepackage{fp}  		
\usetikzlibrary{calc, shapes, positioning}
\usetikzlibrary{matrix,decorations.markings,decorations.fractals}

\newtheorem{theorem}{Theorem}[section]

\newtheorem{lemma}[theorem]{Lemma}

\newtheorem{proposition}[theorem]{Proposition}
\theoremstyle{definition}
\newtheorem{definition}[theorem]{Definition}
\newtheorem{example}[theorem]{Example}
\theoremstyle{remark}
\newtheorem{remark}[theorem]{Remark}

\def\Z{\mathbb{Z}}
\def\N{\mathbb{N}}

\def\R{\mathbb{R}}

\def\A{\mathcal{A}}
\def\G{\mathcal{G}}
\def\HH{\mathcal{H}}
\def\P{\mathcal{P}}
\def\O{\mathcal{O}}
\def\0{\mathbf{0}}
\def\a{\mathbf{a}}
\def\b{\mathbf{b}}
\def\cvec{\mathbf{c}}
\def\e{\mathbf{e}}

\def\u{\mathbf{u}}
\def\vv{\mathbf{v}}
\def\w{\mathbf{w}}
\def\x{\mathbf{x}}
\def\y{\mathbf{y}}
\def\z{\mathbf{z}}

\def\supp{\operatorname{supp}}
\def\climb{\operatorname{climb}}

\title{An Ehrhart theoretic approach to generalized Golomb rulers}
\author[Bogart]{Tristram Bogart}
\address{Departamento de Matem\'aticas \\ Universidad de los Andes \\ Bogot\'a, Colombia}
\email{tc.bogart22@uniandes.edu.co}
\author[Cuellar]{Daniel Felipe Cu\'ellar}
\address{Departamento de Matem\'aticas \\ Universidad de los Andes \\ Bogot\'a, Colombia}
\email{df.cuellar@uniandes.edu.co}

\date{\today}

\begin{document}

\begin{abstract}
A Golomb ruler is a sequence of integers whose pairwise differences, or equivalently pairwise sums, are all distinct. This definition has been generalized in various ways to allow for sums of $h$ integers, or to allow up to $g$ repetitions of a given sum or difference. Beck, Bogart, and Pham applied the theory of inside-out polytopes of Beck and Zaslavsky to prove structural results about the counting functions of Golomb rulers. We extend their approach to the various types of generalized Golomb rulers. 
\end{abstract}

\maketitle

\section{Introduction} \label{sec:intro}
A \emph{Golomb ruler} of length $t$ with $m+1$ markings is a sequence of integers $0 = x_0 < x_1 < \dots < x_{m-1} < x_m = t$ such that the differences $x_j - x_i$ are all distinct. Golomb rulers were originally considered by Sidon \cite{Sidon} and are also known as \emph{Sidon sets} or \emph{$B_2$-sets}. 

The main question that has been studied about such sets is their maximum density; that is, for a given $t$, what is the maximum possible $m$ such that there exists a Golomb ruler of length $t$ with $m+1$ markings? It has long been known that as $t$ tends to $\infty$, this maximum is asymptotic to $\sqrt{t}$. The lower bound is due to Singer \cite{Singer} and the asymptotically matching upper bound to Erd\'os and Tur\'an \cite{ET}. For a range of more recent and related results, we refer the reader to O'Bryant's survey \cite{O'Bryant}. 

A different problem is to count Golomb rulers given the parameters $m$ and $t$. We denote by $b_m(t)$ the number of Golomb rulers of length $t$ with $m+1$ markings. For $m$ fixed and $t$ tending to infinity, almost every possible sequence is a Golomb ruler, so
\[ \lim_{t \to \infty} \frac{b_m(t)}{\binom{t-1}{m-1}} = 1.\]
In order to obtain more algebraically precise results, Beck, Bogart, and Pham \cite{BBP} introduced the following framework. Since $x_0 = 0$, the sequence $\x = (x_0, x_1, \dots, x_m)$ is equivalently specified by the sequence of successive differences $\z = (z_1, \dots, z_m)$, where $z_i = x_i - x_{i-1}$. The condition that $x_0 < x_1 < \dots < x_m$ becomes the condition that the $z_i$'s are all positive, and the $z_i$'s, like the $x_i$'s, must all be integers. That is, each Golomb ruler corresponds to a lattice point in the $t$th dilation of the standard $(m-1)$-simplex in $\R^m$. Finally, the condition that differences are unique can be written as a collection of linear inequations in the $z_i$'s.

 In more general terms, we have just observed that Golomb rulers are in bijection with the lattice points in the interior of a certain polytope that do not lie on any of the hyperplanes in a certain arrangement. Beck and Zaslavsky \cite{BZ} studied such sets of lattice points in general under the name of \emph{inside-out polytopes}. They extended Ehrhart's theorem (both quasipolynomiality and reciprocity) to this context. Their theory could thus be brought to bear on the counting function $b_m(t)$ (see Theorem \ref{thm:BBP1} below.) It also yielded a bijection between the combinatorial types of Golomb rulers and the regions of the hyperplane arrangement that intersect the interior of the polytope.  

 Golomb rulers have been generalized in several ways.
\begin{definition} \label{def:rulers}
        We define a \emph{ruler} to be any sequence of integers $0=x_0<x_1< \cdots <x_{m-1}<x_{m}=t$. 
        \begin{enumerate}
            \item The ruler is a \emph{$B_2[g]$-set} if each positive integer admits at most $g$ representations as a sum of two markings. 
            \item The ruler is a \emph{$g$-Golomb ruler} or \emph{$B_2^{-}[g]$-set} if each positive integer admits at most $g$ representations as a difference of two markings. 
            \item For $h\geq 2$, the set is a \emph{$B_h$-set} if each positive integer admits at most one representation as a sum of $h$ markings.
            \item Combining the first and third definitions, the ruler is a \emph{$B_h[g]$-set} if each positive integer admits at most $g$ representations as a sum of $h$ markings.
           \end{enumerate}  
\end{definition}

\begin{remark} A Golomb ruler (i.e. $B_2^-[1]$ set) is the same as a $B_2$-set, but for $g > 2$, a $B_2[g]$-set is not the same as a $B^{-}_2[g]$-set.
\end{remark}

The asympototic density question is still open in these cases. To our knowledge, the best general upper and lower bounds for the density of $B_h[g]$ sets appear in a recent preprint of Johnson, Tait, and Timmons \cite{JTT}. They also note that in many cases, their upper bound matches one of Green \cite{Green}, and their lower bound matches one of Caicedo, G\'omez, G\'omez, and Trujillo \cite{CGGT}. For $B_h$-sets, Dellamonica, Kohayakwawa, Lee, R\"odl, and Samotij \cite{DKLRS} give good asymptotic results for the total number of $B_h$ sets of length at most $t$. 

Returning to the question of enumeration, given $m$ and $t$, we denote by $b_{m}[g](t)$, $b_m^-[g](t)$, and $b_{m,h}(t)$
the number of rulers of length $t$ with $m+1$ markings that are respectively $B_2[g]$-sets. $B_2^-[g]$-sets, and $B_h$-sets. (We will not consider $B_h[g]$-sets from now on, but it should be possible to extend our results on $B_2[g]$-sets to this case.) In comparison with \cite{DKLRS}, our approach deals with the simpler situation in which the number of markings $m+1$ is fixed, but it yields results that are more precise in an algebraic sense.   

   
   In Section \ref{sec:background} we state our main theorems about the counting functions of generalized Golomb rulers after laying out the necessary definitions. In Section \ref{sec:inside-out}, we extend the approach of Beck, Bogart, and Pham in order to prove these theorems. In all cases, the underlying polytope is still a dilated standard simplex. For $B_h$-sets, we explicitly describe the hyperplane arrangement that must be removed. For the other types of generalized Golomb rulers, the conditions yield that not a hyperplane arrangement but a subspace arrangement must be removed, and we explicitly describe these arrangements. Some of the results of Beck and Zaslavsky extend to the situation of subspace arrangements, so in particular we obtain quasipolynomiality results in all cases. 

  In addition, Beck, Bogart, and Pham gave a combinatorial interpretation of the regions of the inside-out polytope in terms of orientations of a certain mixed graph that satisfy a certain coherence property. Unfortunately, in the course of this project we realized that this result is not correct. There is indeed an explicit injection from regions to orientations but this function is \emph{not} surjective in general. We review the proof of injectivity and give an explicit counterexample to surjectivity in Section \ref{sec:counterexample}. Finally, in Section \ref{sec:graphs} we extend the construction and the injective map to the case of $B_h$-sets, using a more elaborate mixed graph.    

\section{Background and Statements of Theorems} \label{sec:background}
Let $\x  = (x_0, \dots, x_m)$ denote a ruler with $m+1$ markings and again identify the ruler with  the sequence of successive differences $\z = (z_1, \dots, z_m)$ of $\x$. As we have seen, these differences form a sequence of positive integers that sum to $t$, which can be identified with an integer point in the interior of the $t$-th dilation of the standard $(m-1)$-simplex. That is,
\[ \{ \textup{rulers with $m+1$ markings and of length $t$}\} \Leftrightarrow t \Delta_{m-1}^\circ \cap \Z^m, \]
where $\Delta_{m-1} = \{(\z: \, z_1, \dots, z_m \geq 0, \sum_{i=1}^m z_i = 1$.

Now by definition, a Golomb ruler is a ruler for which the differences $x_j - x_i$ between two markings are all distinct. In terms of the consecutive differences, we have $x_j - x_i = \sum_{k=i+1}^j z_k$. Thus a ruler is a Golomb ruler if for every pair of consecutive subsets $U, V$ of $[m]$ we have 
\[ \sum_{i \in U} z_i \neq \sum_{i \in V} z_i.\]

\begin{definition} In the following, we specify rulers by their consecutive differences. 
  \begin{enumerate} 
 \item    Two Golomb rulers $\z = (z_1, \dots, z_m)$ and $\w = (w_1, \dots, w_m)$ are \emph{combinatorially equivalent} if for all consecutive sets $U,V\subseteq [m]$, we have
 \begin{equation} \label{eq:equivalence} \sum_{i\in U}z_i<\sum_{i\in V}z_i \iff \sum_{i\in U}w_i<\sum_{i\in V}w_i. \end{equation}
  \item Applying the same notion of equivalence to Golomb rulers with \emph{real} entries, the \emph{multiplicity} of an (integral) ruler $\z$ is defined as the number of combinatorial types of real Golomb rulers in a sufficiently small neighborhood of $\z$.
  \end{enumerate}
\end{definition}
Note that if $\z$ is itself a Golomb ruler, then its multiplicity is one.

\begin{theorem}\cite[Theorem 1]{BBP} \label{thm:BBP1}
  The Golomb counting function $b_m(t)$, is a quasipolynomial of degree in $t$ of degree $m-1$, with leading coefficient $\frac{1}{(m-1)!}$. The evaluation $(-1)^{m-1}b_m(-t)$ equals the number of rulers with length $t$ and $m+1$ markings, counted with multiplicity, and the evaluation $(-1)^{m-1}b_m(0)$ equals the number of combinatorial types of Golomb rulers with $m+1$ markings.
\end{theorem}

We extend parts of this result to generalized Golomb rulers as follows.

\begin{theorem} \label{thm:quasipolyn} For any $m \geq 1$, and (where appropriate) $h \geq 2$, $g \geq 2$, the functions $b_{m}[g](t)$, $b_m^-[g](t)$ and $b_{m,h}(t)$
  are all quasipolynomials in $t$ of degree $m-1$ with leading coefficient $\frac{1}{(m-1)!}$. 
\end{theorem}

 Our remaining results apply only to $B_h$-sets. The reason for this is that (as we will see in Section \ref{sec:inside-out}), $B_h$-sets are defined by avoidance of certain \emph{hyperplanes}, while the other types of generalized Golomb rulers are defined by avoidance of certain \emph{subspaces}.

 By definition, a $B_h$ set with $m+1$ markings is a ruler $\x$ such that for each pair of distinct sequences $0 \leq r_1 \leq \dots \leq r_h \leq m$ and $0 \leq s_1 \leq \dots \leq s_h \leq m$ we have
 \begin{equation} \label{eq:Bh-constraint} x_{r_1} + \dots + x_{r_h} \neq x_{s_1} + \dots + x_{s_h}. \end{equation}

 \begin{definition}
   \begin{enumerate}
     \item Two such rulers are \emph{combinatorially equivalent} if for every such pair $0 \leq r_1 \leq \dots \leq r_h \leq m$ and $0 \leq s_1 \leq \dots \leq s_h \leq m$, either the inequality $x_{r_1} + \dots + x_{r_h} < x_{s_1} + \dots + x_{s_h}$ holds for both rulers or the opposite inequality $x_{r_1} + \dots + x_{r_h} > x_{s_1} + \dots + x_{s_h}$ holds for both rulers.

     \item We can apply the same notion of equivalence for \emph{real} Golomb rulers $0=y_0 < y_1 < \dots < y_{m-1} < y_m = t$. Then the \emph{$B_h$-multiplicity} of an integral ruler $0=x_0 < x_1, \dots < x_{m-1} < x_m = t$ is the number of combinatorial types of real Golomb ruler in an $\epsilon$-neighborhood of $\x$ for sufficiently small $\epsilon$.
   \end{enumerate}
   \end{definition}

   \begin{theorem} \label{thm:reciprocity}
  Let $h \geq 2$ and $m \geq 1$. 
    \begin{enumerate}
    \item For each $t > 0$, the evaluation $(-1)^{m-1}b_{m,h}(-t)$ equals the total number of rulers with length $t$ and $m+1$ markings, counted with $B_h$-multiplicities.
    \item The evaluation $(-1)^{m-1}b_{m,h}(0)$ is the number of combinatorial types of $B_h$-sets.
    \end{enumerate}
  \end{theorem}


\section{Inside-out polytopes and proofs of Theorems \ref{thm:quasipolyn} and \ref{thm:reciprocity}} \label{sec:inside-out}
Our results are proved via the theory of \emph{inside-out polytopes}, developed by Beck and Zaslavsky \cite{BZ}. Their main ideas, which we now review, extend Ehrhart theory to the situation of a polytope with a hyperplane arrangement or a subspace arrangement removed.

An \emph{inside-out polytope} in $\R^d$ is a pair $(P, \HH)$ where $P$ is a rational polytope and $\HH$ is a rational hyperplane arrangement. A \emph{region} of $(P, \HH)$ is a connected component of $\P \setminus \bigcup_{H \in \HH} H$ and a \emph{closed region} (respectively \emph{open region}) is simply the relative closure (respectively relative interior) of a region. For $\x \in \R^d$, the \emph{multiplicity} $m_{P.\HH}$ of $\x$ with respect to $(P,\HH)$ is defined to be the number of closed regions that contain $\x$. In particular, if $\x$ does not belong to $P$ then $m_{P, \HH}(\x) = 0$ and if $\x$ is contained in an open region of $(P, \HH)$  then $m_{P, \HH}(\x) = 1$.

By analogy with standard Ehrhart theory (see for example \cite{BR}), the \emph{closed Ehrhart function} and the \emph{open Ehrhart function} of the inside-out polytope $(P,\mathcal{H})$ are respectively defined to be 
    \[ E_{P,\mathcal{H}}(t):=\sum_{x\in t^{-1}\mathbb{Z}^d}m_{P,\mathcal{H}}(x)\text{, and }\]
    \[ E^{\circ}_{P,\mathcal{H}}(t):=\#(t^{-1}\mathbb{Z}^d\cap [P\backslash \cup_{H\in\mathcal{H}}H])\]
    
\begin{theorem}\cite[Theorem 4.1]{BZ} \label{thm:inside-out}
        If $(P,\mathcal{H})$ is an inside-out polytope in $\mathbb{R}^d$ such that $\mathcal{H}$ does not contain the degenerate hyperplane $\mathbb{R}^d$, then both $E_{P,\mathcal{H}}(t)$ and  $E^{\circ}_{P^{\circ},\mathcal{H}}(t)$ are quasipolynomials in $t$ of degree $d$, with leading coefficients equal to the volume of $P$ and periods dividing the least common multiple of the denominators in the coordinates of the vertices of the closed regions of $(P,\mathcal{H})$. The value $E_{P,\mathcal{H}}(0)$ is equal to the number of closed regions of $(P,\mathcal{H})$. Furthermore, the Ehrhart reciprocity formula 
        \[ E^{\circ}_{P^{\circ},\mathcal{H}}(t)=(-1)^dE_{P,\mathcal{H}}(-t) \]
        continues to hold in this case.
\end{theorem}

Note that the Ehrhart reciprocity law involves not the open Ehrhart function $E^{\circ}_{P,\HH}$ itself, but the function $E^\circ_{P^\circ,\HH}$ for which we remove not only the hyperplanes of $\HH$, but also the boundary of $P$. If the facet-defining hyperplanes of $P$ belong to $\HH$, then these two functions coincide.  

More generally, if $\A$ is a rational affine subspace arrangement in $\R^d$ and $P$ is a rational $d$-polytope, then the closed Ehrhart function $E_{P,\A}$ and open Ehrhart function $E^\circ_{P,\A}$ can be defined just as above. However, $\A$ no longer divides $P$ into regions and so multiplicity must be defined in a different way.  The multiplicity of $x\in P$ is defined to be
\[ m_{P,\mathcal{A}}(x)=\sum_{u\in \mathcal{L}(\mathcal{A}):x\in u}\mu(\hat{0},u)(-1)^{codim(u)} \]
where $\mathcal{L}(\A)$ is the intersection semilattice of flats of $\A$. If $x \notin P$ then we simply define its multiplicity to be zero. If $\A$ is in fact a hyperplane arrangement then this algebraic definition of multiplicity coincides with the geometric one given above \cite[Lemma 3.4]{BZ}.

\begin{theorem}\cite[Theorem 8.2]{BZ} \label{thm:inside-out-subspace}
        Let $P$ be a rational polytope of dimension in $\mathbb{R}^d$ and $\mathcal{A}$ a rational subspace arrangement. Then, $E_{P,\mathcal{A}}(t)$ and $E^{\circ}_{P^{\circ},\mathcal{A}}(t)$ are quasipolynomials in $t$ of degree $\dim(P)$, with leading coefficients equal to the volume of $P$, periods dividing the least common multiple of the denominators in the coordinates of the vertices of $(P,\mathcal{H})$ (which are the vertices of $P$ and its regions), and the flats of dimension $0$ in  $\mathcal{L}(\mathcal{A})$. Furthermore, we have the reciprocity law
        \[ E^{\circ}_{P^{\circ},\mathcal{A}}(t)=(-1)^dE_{P,\mathcal{A}}(-t).\]
\end{theorem}

Theorems \ref{thm:quasipolyn} and \ref{thm:reciprocity} will thus follow by interpreting the various types of generalized Golomb rulers as lattice points in inside-out polytopes. From their definitions via linear inequations it is not surprising that this will be possible, but we give explicit combinatorial descriptions of the hyperplanes and subspaces in order to be able to calculate examples.

\subsection{$B_h$-sets}
\begin{proposition} \label{prop:Bh-description}
  $B_h$-sets with $m+1$ markings and length $t$ are in bijection with lattice points in the interior of $t \Delta_{m-1}$
  that are not contained in any nondegenerate hyperplane of the form 
   \begin{equation} \label{eqn:Bh-z-form} \sum_{k=1}^\ell \left( \sum_{j \in U_k} z_j \right) = \sum_{\ell + 1}^{h'} \left( \sum_{j \in U_k} z_j
  \right), \end{equation}
  where $h' \leq h$ and $U_1, \dots, U_{h'}$ are proper consecutive subsets of $[m]$.  
\end{proposition}

\begin{proof}
  Consider one of the constraints on $B_h$-sets given in (\ref{eq:Bh-constraint}); that is,
   \[ x_{r_1} + \dots + x_{r_h} \neq x_{s_1} + \dots + x_{s_h}\]
   with $0 \leq r_1 \leq \dots \leq r_h \leq m$ and $0 \leq s_1 \leq \dots \leq s_h \leq m$. Rewrite the constraint as
   $\sum_{j=1}^h x_{r_j} - x_{s_j} \neq 0$.
 In terms of the consecutive differences $z_i = x_i - x_{i-1}$, this becomes
  \[ \sum_{j: r_j > s_j} \left( \sum_{i = s_j+1}^{r_j} z_i \right) = \sum_{j: r_j < s_j} \left( \sum_{i = r_j+1}^{s_j} z_i \right) \]
  which is an inequation of the desired form with $\ell = \left| \{j: \, r_j > s_j \} \right|$ and $h' = \left| \{j: \, r_j \leq s_j \} \right|$. Conversely, each inequation in $\z$ as in (\ref{eqn:Bh-z-form}) gives a valid inequation on $\x$ by reversing these steps.  
\end{proof}

 \begin{example} \label{ex:B3-4markings}
   The $B_3$-sets with 4 markings and length $t$ are in bijection with lattice points of $t \Delta^{\circ}_3$ that avoid the 20 hyperplanes
   
   \begin{table}[ht]
            \centering
            \begin{tabular}{ccc}
                $z_1=z_2$ & $2z_1=z_2$ &$z_1=2z_2$\\
                $z_1=z_3$ & $2z_1=z_3$ &$z_1=2z_3$\\
                $z_2=z_3$ & $2z_2=z_3$ &$z_2=2z_3$\\
                $z_1=z_2+z_3$ & $2z_1=z_2+z_3$ &$z_1=2z_2+z_3$\\
                $z_1=z_2+2z_3$ & $z_1=2z_2+2z_3$ & $z_1+z_2=z_3$\\
                $2z_1+z_2=z_3$ & $z_1+2z_2=z_3$ &$2z_1+2z_2=z_3$\\
                $z_1+z_2=2z_3$ & $z_1+z_3=z_2$ &\\
            \end{tabular}
   \end{table}
   
   as shown in Figure \ref{fig:B3}. There are 37 intersection points in the interior of the simplex, 12 points of intersection in the boundary of the polytope and the hyperplanes divide the polytope into 80 closed regions. From Theorem \ref{thm:quasipolyn}, we conclude that the Ehrhart quasipolynomial $b_{4,3}(t)$ has period $2520$, so it is not practical to describe it completely. However, we calculate\footnote{https://github.com/Quillar/ExtensionsGolomb/tree/a97554c93916108bfb69efb5326106442ed7e11f}
   that if $t$ is a multiple of 2520, the number of $B_3$-sets with 4 markings and length $t$ is
   \[ b_{4,3}(t) = \frac{1}{2}t^2-\frac{55}{6}t+80 .\]
\end{example}

\begin{remark} \label{remark:non-uniqueness}The correspondence given in the proof of Proposition \ref{prop:Bh-description} is not a bijection: several collections of consecutive subsets may correspond to the same constraint. For example, again take $B_3$-sets with 4 markings and consider the equation $x_2+x_3+x_3=x_0+x_1+x_4$ that must be avoided. Following the proof of Proposition \ref{prop:Bh-description}, we rewrite this equation as $(x_2-x_0)+(x_3-x_1)=x_4-x_3$. In terms of $\z$, this becomes $(z_1+z_2)+(z_2+z_3)=z_4$. However, we can also write the original equation as $(x_2-x_1)+(x_3-x_0)=x_4-x_3$, which yields $z_2+(z_1+z_2+z_3)=z_4$.
\end{remark}

\begin{figure}
\includegraphics[scale=1.2]{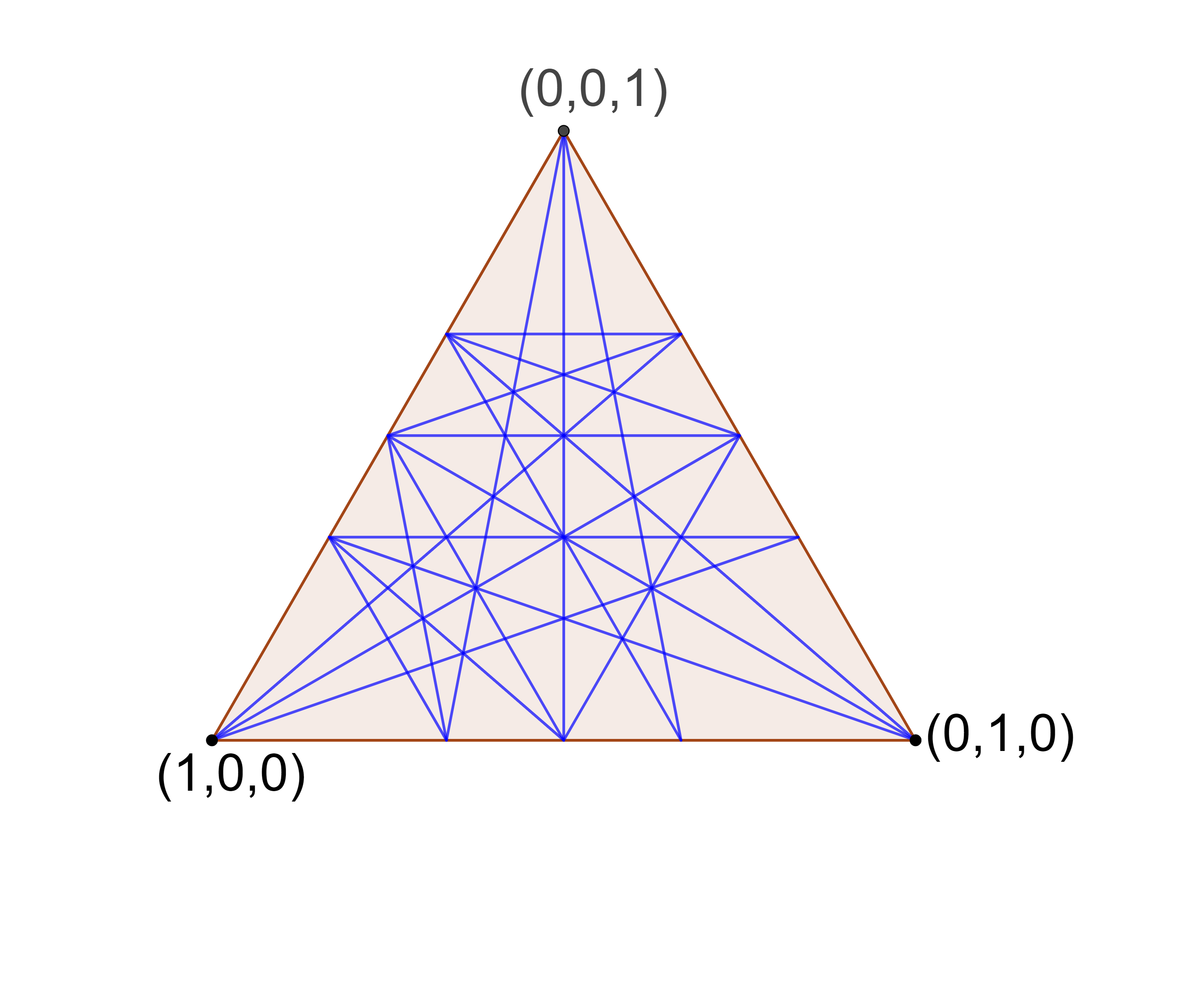}
\caption{The inside-out-polytope for $B_3$-sets with 4 markings.}  \label{fig:B3}
\end{figure}

\begin{proof}[Proof of Theorem \ref{thm:quasipolyn} for $B_h$-sets and Theorem \ref{thm:reciprocity}] Let $\HH_{m,h}$ be the hyperplane arrangement described in Proposition \ref{prop:Bh-description}. By this proposition, we have that $b_{m,h}(t) = E^{\circ}_{\Delta_m^\circ,\HH_{m,h}}(t)$.   It is immediate from Theorem \ref{thm:inside-out} that the function $b_{m,h}$ is a quasipolynomial. The same theorem yields that $b_{m,h}(0)$ is the number of closed regions of the inside-out polytope $(\Delta_m, \HH_{m,h})$, which is the number of possible combinatorial types of a $B_h$-set. (For sufficiently large $t$, there will exist a lattice point in each open region, so all of these types are actually realized. This holds because $\Delta_m$ is unimodularly equivalent to a full-dimensional polytope in $\R^{m-1}$ via projection on any $m-1$ coordinates.) Finally, reciprocity yields that for $t > 0$,
  \[ b_{m,h}(-t) = (-1)^{m-1}E_{\Delta_m, \HH_{m,h}}(t) \]
  which is the number of rulers counted with their $\HH_{m,h}$-multiplicities. 
\end{proof}

\subsection{$B_2[g]$-sets} $ $ \newline

By Definition \ref{def:rulers}, a ruler $0=x_0<x_1< \cdots <x_{m-1}<x_{m}=t$ is a $B_2[g]$-set if it does not satisfy any chain of equations of the form
\begin{equation} \label{eq:B2g-chain} x_{\ell_0}+x_{r_0}=x_{\ell_1}+x_{r_1}= \cdots =x_{\ell_{g-1}}+x_{r_{g-1}}=x_{\ell_{g}}+x_{r_{g}}.
\end{equation}
As we did for $B_h$-sets, we can obtain an inside-out polytope by expressing these equations in terms of the successive differences $z_1, \dots, z_m$.

\begin{proposition} \label{prop:B_2[g]-description}
$B_2[g]$-sets with $m+1$ markings and length $t$ are in bijection with lattice points in the interior of $t \Delta_m$ that are not contained in any subspace of the form 
 \begin{equation} \label{eq:B2g-subspace} \sum_{r_g}^{r_0}z_i=\sum_{\ell_0}^{\ell_1} z_i+\sum_{r_g}^{r_1}z_i= \cdots =\sum_{\ell_0}^{\ell_k}z_i+\sum_{r_g}^{r_k}z_i= \cdots =\sum_{\ell_0}^{\ell_g}z_i \end{equation}
 where the indices satisfy
 \begin{equation} \label{eq:indexorder} 0\leq \ell_0<\ell_1< \cdots <\ell_{g}\leq r_{g}<r_{g-1}< \cdots <r_0\leq m.
 \end{equation}
\end{proposition}

\begin{proof}
  Given indices $i < j \leq k < \ell$, and any ruler $\x$, we have $x_i+x_j<x_i+x_k<x_j+x_\ell \leq <x_k+x_\ell$ by the order of the markings, so it can never be the case that $x_i+x_k=x_j+x_\ell$ nor that $x_i+x_j=x_k+x_\ell$. So in order to see whether $\x$ is a $B_g$-set, the only equation on these four indices that must be considered is $x_i+x_\ell=x_j+x_k$. 

  Now consider a chain of equations as in (\ref{eq:B2g-chain}). Without loss of generality $\ell_i \leq r_i$ for each $i$ and $\ell_0 < \ell_1 < \dots \ell_g$. Then by the previous paragraph, we may also assume that $r_0 > r_1 > \dots r_g$. That is, the indices satisfy (\ref{eq:indexorder}). Since $x_j = \sum_{i=1}^j z_i$ (using the usual convention that the empty sum equals zero to correctly obtain $x_0=0$), (\ref{eq:B2g-chain}) is equivalent to 
  \[ \sum_{i=1}^{\ell_0}z_i+\sum_{i=1}^{r_0}z_i= \cdots =\sum_{i=1}^{\ell_k}z_i+\sum_{i=1}^{r_k}z_i= \cdots =\sum_{i=1}^{\ell_g}z_i+\sum_{i=1}^{r_g}z_i .\]
  Now cancel the common terms $\sum_{i=1}^{\ell_0} z_i$ and $\sum_{i=1}^{r_g}z_i$ to obtain (\ref{eq:B2g-subspace}).

  By reversing this process, a chain of equations of the form (\ref{eq:B2g-subspace}) also yields one of the form (\ref{eq:B2g-chain}).  
\end{proof}

\begin{example}
  For $B_2[2]$-sets with five markings (that is, $g=2$ and $m=4$), the only chain of equations we must consider is
  \[ x_0+x_4=x_1+x_3=x_2+x_2. \]
  In terms of the successive differences, this becomes
  \[ z_1+z_2+z_3+z_4=2z_1+z_2+z_3=2z_1+2z_2, \]
  or after cancelling common terms,
  \[ z_3+z_4=z_1+z_3=z_1+z_2. \]
This forbidden subspace is a line which passes through the interior of the tetrahedron $\Delta_3$.
\end{example}

\textit{Proof of Theorem \ref{thm:quasipolyn} for $B_2[g]$-sets:} This follows immediately from Theorem \ref{thm:inside-out-subspace} and Proposition \ref{prop:B_2[g]-description}.

\subsection{$B_2^-[g]$-sets}
The defining inequations for $B_2^-[g]$-sets are similar to those for $B_2[g]$-sets, but with differences instead of sums. That is, we must avoid chains of equations
\begin{equation} \label{eq:B2gdifference-chain} x_{r_0}-x_{\ell_0}=x_{r_1}-x_{\ell_1} = \cdots =x_{r_{g-1}}-x_{\ell_{g-1}}=x_{r_{g}}-x_{\ell_{g}}.
\end{equation}

Again we obtain an inside-out polytope by expressing these chains in terms of the successive differences.

\begin{proposition} \label{prop:B_2^-[g]-description}
  $B_2[g]$-sets with $m+1$ markings and length $t$ are in bijection with lattice points in the interior of $t \Delta_{m-1}$ that are not contained in any subspace of the form
  \begin{equation} \label{eq:B2gdifference-subspace} \sum_{i \in U_0}z_i=\sum_{i \in U_1}z_i = \cdots = \sum_{i \in U_g}z_i \end{equation}
  where $U_0, \dots, U_g$ are consecutive subsets of $[m]$, none of which is contained in another.
\end{proposition}

\begin{proof}
As in the proof of Proposition \ref{prop:Bh-description}, we must rewrite each constraint given by (\ref{eq:B2gdifference-chain}) in terms of the vector $\z$ of successive differences. Now for any $\ell < r$ we have $x_r - x_\ell = (z_{\ell + 1} + z_{\ell + 2} + \dots + z_r$, so we obtain a constraint of the form (\ref{eq:B2gdifference-subspace}) by taking $U_j = \{\ell_{j+1}, \ell_{j+2}, \dots, r_j\}$ for $j=0,1,\dots,g$. If $U_j$ is contained in $U_k$, then the condition is redundant because each succesive difference is positive, so it can never be the case that $\sum_{i \in U_j}z_i=\sum_{i \in U_k}z_i$. By reversing this process, we can also transform any constraint of the form (\ref{eq:B2gdifference-subspace}) into one of the form (\ref{eq:B2gdifference-chain}). 
\end{proof}  

\begin{remark} Unlike the original case of Golomb rulers (i.e. $B_2^{-}[1]$-sets), we cannot restrict to conditions given by \emph{disjoint} consecutive sets $U_0, \dots, U_g$. For example, let $g=2$ and $m=5$ and consider the chain of equations
\[ x_3 - x_0 = x_4 - x_1 = x_5 - x_2.\]
In terms of the $z_i$'s, this becomes
\[z_1 + z_2 + z_3  = z_2 + z_3 + z_4 = z_3 + z_4 + z_5\]
so that the consecutive subsets are $U_0 = \{1,2,3\}$, $U_1 = \{2,3,4\}$ and $U_2 = \{3,4,5\}$ which are not disjoint. On the other hand, we could obtain an equivalent chain of equations with disjoint sets by cancelling the common term $z_3$, but then one of these sets would be $U_1 \setminus \{3\} = \{2, 4\}$, which is no longer consecutive.
\end{remark}

  \textit{Proof of Theorem \ref{thm:quasipolyn} for $B_2^-[g]$-sets:}
  This follows immediately from Theorem \ref{thm:inside-out-subspace} and Proposition \ref{prop:B_2^-[g]-description}.

\section{Combinatorial Types and Orientations of Mixed Graphs} \label{sec:counterexample}
Returning to the original case of Golomb rulers, the combinatorial types are related to acyclic orientations of a certain mixed graph. We have already seen that combinatorial types are in bijection with open regions of an inside-out polytope $(\Delta_{m-1}, \G_{m})$, where $\G_m$ is the hyperplane arrangement given by the hyperplanes (\ref{eq:equivalence}) for each pair of consecutive subsets of $[m]$, so we will work directly with these regions. 

\begin{definition}
  \begin{enumerate}
\item The \emph{Golomb graph} $\Gamma_{m}$ is a mixed graph whose vertices are the proper consecutive subsets of $[m]$. The graph is complete and an edge is directed from $U$ to $V$ if $U\subseteq V$. All other edges are undirected. 
\item An orientation of $\Gamma_m$ is called \emph{coherent} if:
  \begin{enumerate}
  \item it is consistent with the directed edges, and
  \item if $U, V, W$ are disjoint consecutive sets and $A = U \cup W$ and $B = C \cup W$ are also consecutive, then
    $A \rightarrow B$ if and only if $U \rightarrow V$.
    \end{enumerate}
  \end{enumerate} 
  \end{definition}

We first review the construction in \cite{BBP} of an injective map $\phi$ from regions of the Golomb inside-out polytope $(\Delta_{m-1}, \G_m)$ to acyclic orientations of the Golomb (mixed) graph $\Gamma_m$. To do this, let $R$ be a region of the inside-out polytope. Then $R$ is the intersection of half-spaces given by inequalities of the form
\[ \sum_{j \in U} z_j < \sum_{j \in V} z_j \]
where $U$ and $V$ range over all pairs of consecutive subsets of $[m]$. We may assume that both sets are proper, since otherwise the inequality would hold over the entire positive orthant (in particular, over all of $\Delta_m$.) Furthermore, we may restrict to \emph{disjoint} sets because if $A$ and $B$ are consecutive subsets that are not disjoint, then $U := A \setminus B$, $V := B \setminus A$, and $W := A \cap B$ are all consecutive sets, and 
\[ \sum_{j \in A} z_j - \sum_{j \in B} z_j = \sum_{j \in U} z_j - \sum_{j \in V} z_j. \]

For a pair of disjoint proper consecutive sets (dpcs) $(U,V)$, orient the edge $UV$ of $\Gamma_m$ by $U \rightarrow V$. If $W$ is disjoint from both $U$ and $V$ and if $A = U \cup W$ and $B = V \cup W$ are consecutive sets, then orient the edge $AB$ by $A \rightarrow B$. Thus the orientation is coherent. It is total because (again) any pair of consecutive sets $A, B$ can be decomposed in this way. Finally, it is acyclic because if there were a cycle then we could take the multiset union $M$ of the sets along the whole cycle and conclude that
$\sum_{j \in M} z_j < \sum_{j \in M} z_j$, which is impossible.

\begin{proposition}
The injection $\phi$ is not surjective when $m=5$. In particular, there is no bijection between the orientations of $\Gamma_5$ and the regions of $(\Delta_4, \G_5)$.
\end{proposition}

\begin{remark}
It would not be difficult to extend the counterexample to any $m \geq 5$.)
\end{remark}

\begin{proof} Consider the following linear order on proper consecutive subsets of $[5]$:
\[ 3 \rightarrow 5 \rightarrow 1 \rightarrow 4 \rightarrow 2 \rightarrow 34 \rightarrow 23 \rightarrow 12 \rightarrow 45 \rightarrow 123 \rightarrow 345 \rightarrow 234 \rightarrow 2345 \rightarrow 1234.\]
Its transitive closure is an acyclic orientation $\O$ on the Golomb graph $\Gamma_5$ consistent with the directed edges given by set containment. To verify that $\O$ is coherent, note that the linear order always places smaller subsets before larger ones, so it suffices to consider pairs $(A,B)$ of sets of the same size. So we check all such pairs, and indeed:
\[ \begin{tabular}{c|c}
  $34 \rightarrow 23$ and $4 \rightarrow 2$ & $34 \rightarrow 45$ and $3 \rightarrow 5$ \\
  $23 \rightarrow 12$ and $3 \rightarrow 1$ & $123 \rightarrow 345$ and $12 \rightarrow 45$ \\
  $123 \rightarrow 234$ and $1 \rightarrow 4$ & $345 \rightarrow 234$ and $5 \rightarrow 2$ \\
  $2345 \rightarrow 1234$  and $5 \rightarrow 1$ &
  \end{tabular}.\]

However, suppose there exists a region $R$ of the inside-out polytope such that $\phi(R) = \O$. Let $\z = (z_1, z_2, z_3, z_4, z_5)$ be a vector in the relative interior of $R$. Then from the edges $34 \rightarrow 23$, $12 \rightarrow 45$, and $5 \rightarrow 1$, we obtain that
\[ z_3 + z_4 < z_2 + z_3, \; z_1 + z_2 < z_4 + z_5, \; z_5 < z_1 \]
and these three inequalities sum to
\[ z_1 + z_2 + z_3 + z_4 + z_5 < z_1 + z_2 + z_3 + z_4 + z_5 \]
which is a contradiction. 
\end{proof}

\section{Reciprocity and the $B_h$-graph} \label{sec:graphs}
In this section we generalize the injection described in Section \ref{sec:counterexample} to the case of $B_h$-sets.

For Golomb graphs, the vertices are consecutive subsets of $[m]$. Here we would need to consider \emph{collections} of such sets, but in light of the non-uniqueness illustrated in Remark \ref{remark:non-uniqueness}, it is better to consider \emph{unions with repetition} of such sets, which are multisets. We represent a multiset $A = \{1^{a_1},2^{a_2},\dots, m^{a_{m}}\}$ by the vector $\a = (a_1, a_2, \dots, a_m) \in \N^m$. In this way, multiset union corresponds to addition of vectors, and multiset containment $A \subseteq B$ corresponds to $\a \leq \b$ with respect to the standard partial order on $\N^m$. A single consecutive set $U = \{j, j+1, \dots k\}$ corresponds to a \emph{consecutive vector} $\e_{[j,k]} := \e_j + \e_{j+1} + \dots + \e_k$.

Next, we need to know which pairs of vectors correspond to inequations defining $B_h$-sets. In particular, given a vector $y^\a$, we need to know the minimum number of consective vectors whose sum is $\a$. The following definition and lemma will help us do this.

\begin{definition}
The \emph{climb} of a vector $\a = (a_1, \dots, a_m)$ is $\climb(\a) = \sum_{j=1}^m \climb_j(\a)$, where
  $\climb_i(\a) = \begin{cases} a_{j}-a_{j-1} & \textup{if } a_{j-1} < a_j \\ 0 & \textup{otherwise} \end{cases}$, \newline
  with $a_0$ taken to be 0.      
\end{definition}

\begin{lemma} \label{lem:climb} Let $\a \in \N^m$. The minimum number $r$ of consecutive vectors whose sum is $\a$ equals the climb of $\a$.
\end{lemma}

\begin{proof}
  Induct on $r$. If $r=1$ then $\a$ is a consecutive vector and the result is clear. Otherwise, we can write $\a = \e_{[j,k]} + \b$ for some nonzero vector $\b \in \N^m$. Then $\climb_j(\a)$ equals either $\climb_j(\b)$ (if $b_{j-1} > b_j$) or $\climb_j(\b) + 1$ (if $b_{j-1} \leq b_j$). Similarly, $\climb_{k+1}(\a)$ equals either $\climb_{k+1}(\b) - 1$ (if $b_k < b_{k+1}$) or $\climb_k(\b)$ (if $b_k \geq b_{k+1}$). For all other values of $\ell$, we always have $\climb_\ell(\a) = \climb_\ell(\b)$. In particular, $\climb(\a) \leq \climb(\b) + 1$. It follows by induction that the number of consecutive vectors in any decomposition of $\a$ is at least the climb of $\a$.

  For the other direction, we first observe that if the support of $\a$ is disconnected (that is, if there exist $j_1 < j_2 < j_3$ such that $a_{j_1}, a_{j_3} > 0$ and $a_{j_2} = 0$), then the climb of $\a$ is the sum of the climbs of the connected components. Also, any decomposition of $\a$ into consecutive vectors respects the connected components of $\a$. So it is enough to consider vectors with connected support.

  Let $\a$ be such a vector, so its support is an interval $[j,k]$, and again let $r := \climb(\a)$. Now we can subtract the consecutive vector $\e_{[j,k]}$, and the remaining vector $(0, \dots, a_j-1, a_{j+1}-1, \dots, a_k -1, \dots, 0)$ has climb exactly $r-1$ because the climb at the first step is one less than that of $\a$, and the climb at all other steps equals that of $\a$.  By induction, we conclude that there exists a decomposition of $\y^\a$ into $r$ consecutive vectors. 
\end{proof}

Now the Golomb graph is complete in the sense that there is either a directed or an undirected edge between every pair of vertices. However, the analogous graph that we will define for $B_h$-sets cannot be complete for the following reason.  

\begin{example}
  Again consider $B_3$-sets with four markings as in Example \ref{ex:B3-4markings}. (That is, $m=4-1 = 3$ and $h=3$.) Two of the hyperplanes we consider are given by $2z_1 = z_2$ and $z_1 = 2z_2 + z_3$. That is, our graph must include vertices given by each of the vectors $2\e_1, \e_2, \e_1, 2\e_2+\e_3$ with (undirected) edges between $2\e_1$ and $\e_2$ and between $2\e_1$ and $2\e_2+\e_3$. However, $2z_1 = 2z_2 + z_3$ is \emph{not} a hyperplane in the arrangement $\HH_{3,3}$ because it does not correspond to an equation built from only three consecutive sets. Therefore our graph must not include an edge between $2\e_1$ and $2\e_2+\e_3$.     
\end{example}

On the other hand, it will not suffice to consider only edges between vectors $\a$ and $\b$ such that the sum of the climbs of $\a$ and of $\b$ is at most $h$. The reason is as follows.

\begin{example}
  Again consider $B_3$-sets with four markings. The hyperplanes listed in Example \ref{ex:B3-4markings} require our graph to contain vertices labelled by the vectors $\e_1$, $2\e_1$, $\e_2$, $\e_3$, and $\e_1+\e_3$ with certain undirected edges between them. We also will need directed edges from $\e_1$ to $\e_1\e_3$, from $\e_1$ to $2\e_1$, and from $\e_3$ to $\e_1+\e_3$ to represent the fact that $z_1$ and $z_3$ are always positive.  The resulting mixed subgraph on these five vertices is shown on the left side of Figure \ref{fig:partial-Bh-graph}

  Now the climbs of the vectors $2\e_1$ and $\e_1\e_3$ are both equal to two, so the sum of the two climbs is greater than $h=3$. But if we do not add an edge between these two vertices, then the graph can be acyclically oriented as in Figure \ref{fig:partial-Bh-graph}. In this orientation, the edge $(1)\to(3)$ would be associated with the inequality $z_1<z_3$ and the edges $(1)(3)\to(2)$ and $(2)\to(1)(1)$, by transitivity, would be associated with the inequality $z_3<z_1$.  
\end{example}

In determining which vertices and edges to include in our graph, we must also take into account that climb is not monotonic. For example, the climb of $(1,0,1)$ is two and the climb of $(1,1,1)$ is only one. However, we have the following result for \emph{pairs} of vectors.

\begin{figure}[h]  

  \centering 
  \begin{subfigure}[b]{0.48\linewidth}
    \begin{tikzpicture}   
      \draw (0,0) node [ ] {$(1)$};
            \draw (6,0) node [ ] {$(3)$};
            \draw (3,2.7) node [ ] {$(2)$};
            \draw (0,4) node [ ] {$(1)(1)$};
            \draw (6,4) node [ ] {$(1)(3)$};
            
            \draw [] (0.5,0) to  node [auto] {$ $} (5.5,0);
            \draw [] (0.5,0.5) to  node [auto] {$ $} (2.5,2.5);
            \draw [] (5.5,0.5) to  node [auto] {$ $} (3.5,2.5);
            \draw [] (0.5,3.7) to  node [auto] {$ $} (2.5,3);
            \draw [] (5.5,3.7) to  node [auto] {$ $} (3.5,3);
            \draw [] (0.4,3.6) to  node [auto] {$ $} (5.5,0.3);
            \draw [-stealth](0,0.5) -- (0,3.5);
            \draw [-stealth](0.5,0.3) -- (5.5,3.5);
            \draw [-stealth](6,0.5) -- (6,3.5);
    \end{tikzpicture}%
    \caption{A partial $B_3$-graph} 
  \end{subfigure}
\begin{subfigure}[b]{0.48\linewidth}
  \begin{tikzpicture}  
      \draw (0,0) node [ ] {$(1)$};
            \draw (6,0) node [ ] {$(3)$};
            \draw (3,2.7) node [ ] {$(2)$};
            \draw (0,4) node [ ] {$(1)(1)$};
            \draw (6,4) node [ ] {$(1)(3)$};
            
            \draw [-stealth](0.5,0) -- (5.5,0);
            \draw [-stealth](0.5,0.5) -- (2.5,2.5);
            \draw [-stealth](5.5,0.5) -- (3.5,2.5);
            \draw [-stealth](2.5,3) -- (0.5,3.7);
            \draw [-stealth](5.5,3.7) -- (3.5,3);
            \draw [-stealth](5.5,0.3) -- (0.4,3.6);
            \draw [-stealth](0,0.5) -- (0,3.5);
            \draw [-stealth](0.5,0.3) -- (5.5,3.5);
            \draw [-stealth](6,0.5) -- (6,3.5);
\end{tikzpicture}
\caption{An invalid acyclic orientation} 
\label{fig:partial-Bh-graph}
\end{subfigure}

\end{figure}  

\begin{proposition} \label{prop:climb_bound}
  Let $\a, \b, \cvec \in \N^m$ be vectors such that $\a$ and $\b$ have disjoint supports. Then
\[ \climb(\a + \cvec) + \climb(\b + \cvec) \geq \climb(\a) + \climb(\b).\]
\end{proposition}
\begin{proof}
  We can write $\cvec$ as a sum of consecutive vectors, so by induction it is sufficient to prove the statement in the case that $\cvec $ is itself a consecutive vector. Furthermore, if $\climb(\a + \cvec) \geq \climb(\a)$ and $\climb(\b + \cvec) \geq \climb(\b)$ then the conclusion immediately follows.

  So without loss of generality, let $\cvec = \e_{[j,k]}$ and assume that $\climb(\a + \cvec) < \climb(\a)$. From the proof of Lemma \ref{lem:climb}, this can happen only if $a_{j-1} > a_j$ and $a_k < a_{k+1}$, and in this case $\climb(\a + \e_{[j,k]}) = \climb(\a) - 1$. In particular, this means that both $j-1$ and $k+1$ belong to the support of $\a$. Since $\a$ and $\b$ have disjoint supports, neither $j-1$ nor $k+1$ belong to the support of $\b$. So $b_{j-1} \leq b_j$ and $b_k \geq b_{k+1}$, and again by the proof of Lemma \ref{lem:climb}, $\climb(\b + \e_{[j,k]}) = \climb(\b) + 1$. We conclude that in this case,
  \[ \climb(\a + \e_{[j,k]}) + \climb(\b + \e_{[j,k]}) = \climb(\a) + \climb(\b).\]
\end{proof}

\begin{definition}
  Let $M_h$ be the set of vectors in $\N^m$ of climb at most $h$.
  \begin{enumerate}
  \item Let $\Gamma_{m,h}$ be the undirected graph on the vertex set $M_h$ where $\u\vv$ is an edge if, when we write $\u = \a + \cvec$, $\vv = \b + \cvec$ with $\supp(\a) \cap \supp(\b) = \emptyset$, we have $\climb(\a) + \climb(\b) \leq h$.
  \item An orientation of the edges of $\Gamma_{m,h}$ is called \emph{coherent} if:
    \begin{itemize}
    \item Every edge $\0 \u$ is oriented from $\0$ to $\u$, and
      \item for every $\a$, $\b$ with $\climb(\a) + \climb(\b) \leq h$ and every $\cvec$ such that $\u = \a + \cvec$ and $\vv = \b + \cvec$ belong to $M_h$, the orientation of $\u \vv$ agrees with the orientation of $\a \b$.
    \end{itemize}
    \end{enumerate}
\end{definition}

  Note that $\Gamma_{m,h}$ is a finite graph because no coordinate of any vector in $M_h$ can be greater than $h$.

  Now define a function $\phi$ from the inside-out polytope $Q := \Delta_{m-1} \setminus \bigcup_{H \in \HH_{m,h}} H$ to the set of orientations of $\Gamma_{m,h}$ as follows. Given an integer vector $\z = (z_1, \dots, z_m) \in Q$ and an edge $\u \vv$ of $\Gamma_{m,h}$, choose the orientation
  \[ \u \to \vv \textup{ if } \sum_{i=1}^mu_iz_i < \sum_{i=1}^mv_iz_i.\]

\begin{theorem} \label{thm:newinjection}
The function $\phi$ induces an injection from the open regions of the inside-out polytope $Q$ to the coherent acyclic orientations of the graph $\Gamma_{m,h}$.  
\end{theorem}

\begin{proof}
  Let $\z \in Q$. We first note that since $\u \in \N^n$, every edge of the form $\0 \u$ is oriented $\0 \rightarrow \u$. For each edge $\u \vv$ of the graph, write $\u = \a + \cvec$, $\vv = \b + \cvec$ where $\supp(\a) \cap \supp(\b) = \emptyset$. Since $\u \vv$ is an edge, we have $\climb(\u)+\climb(\vv) \leq h$. Then by Proposition \ref{prop:climb_bound}, $\climb(\a)+\climb(\b) \leq h$ as well. Since $\z$ cannot belong to the hyperplane $\sum_{i=1}^ma_iz_i = \sum_{i=1}^m b_iz_i$ in $\Gamma_{m,h}$, the edge $\a \b$ is assigned an orientation by $\phi(z)$, say $\a \rightarrow \b$. By adding $\sum_{i=1}^m c_iv_i$ to both sides of the inequality $\sum_{i=1}^ma_iz_i < \sum_{i=1}^m b_iz_i$ we see that the edge $\u \vv$ is also oriented as $\u \rightarrow \vv$. That is, for any $\z \in Q$, the orientation $\phi(z)$ is total and coherent.

Furthemore, if the orientation $\phi(\z)$ contained a cycle
  \[ \u^{(1)} \rightarrow \u^{(2)} \rightarrow \dots \rightarrow \u^{(r)} \rightarrow \u^{(1)} \]
  then by the definition of $\phi$ we would have
  \[ \sum_{i=1}^mu^{(1)}_iz_i < \sum_{i=1}^mu^{(2)}_iz_i < \dots \sum_{i=1}^mu^{(r)}_iz_i < \sum_{i=1}^mu^{(1)}_iz_i \]
  which is a contradiction. Thus $\phi(z)$ is acyclic.
  
  It remains to show that $\phi$ is invariant within each region and that it takes different values on different regions. That is:
  
\textbf{Claim:} Given $\z, \w \in Q$, we have $\phi(\w) = \phi(\z)$ if and only if $\z$ and $\w$ lie in the same region of the inside-out polytope.

To prove this claim, first suppose that $\phi(\z) = \phi(w)$. Then for all edges $\u \vv$ in $\Gamma_{m,h}$, we have
  \[ \sum_{i=1}^m u_iz_i < \sum_{i=1}^m v_iz_i \Leftrightarrow \sum_{i=1}^m u_iw_i < \sum_{i=1}^m v_iw_i.\]
  Since edges of $\Gamma_{m,h}$ correspond to hyperplanes in $\G_{m,h}$, $\z$ and $\w$ lie on the same side of each of the hyperplanes in the arrangement. 

  On the other hand, suppose $\phi(\z) \neq \phi(\w)$. Then there exists an edge $\u \vv$ such that when we again write $\u = \a + \cvec$, $\vv = \b + \cvec$ where $\supp(\a) \cap \supp(\b) = \emptyset$, the orientation $\phi(\z)$ has $\a \rightarrow \b$ and $\phi(\w)$ has $\a \leftarrow \b$. That is,
  \[ \sum_{i=1}^m u_i z_i < \sum_{i=1}^m v_iz_i, \; \sum_{i=1}^m u_i w_i > \sum_{i=1}^m v_iw_i.\] 
 This implies that $\w$ and $\z$ lie on opposite sides of the hyperplane in $G_m$ determined by the pair $\a\b$. 
\end{proof}

\section*{Acknowledgements}

Tristram Bogart was supported by internal research grant INV-2020-105-2076 from the Faculty of Sciences of the Universidad de los Andes.

\bibliography{ref}

\newcommand{\etalchar}[1]{$^{#1}$}
\providecommand{\bysame}{\leavevmode\hbox to3em{\hrulefill}\thinspace}
\providecommand{\MR}{\relax\ifhmode\unskip\space\fi MR }
\providecommand{\MRhref}[2]{%
  \href{http://www.ams.org/mathscinet-getitem?mr=#1}{#2}
}
\providecommand{\href}[2]{#2}
\begin{thebibliography}{DJKL{\etalchar{+}}16}

\bibitem[BBP12]{BBP}
Matthias Beck, Tristram Bogart, and Tu~Pham, \emph{Enumeration of golomb rulers
  and acyclic orientations of mixed graphs}, The Electronic Journal of
  Combinatorics \textbf{19} (2012), no.~3, P42.

\bibitem[BR07]{BR}
Matthias Beck and Sinai Robins, \emph{Computing the continuous discretely},
  vol.~61, Springer, 2007.

\bibitem[BZ06]{BZ}
Matthias Beck and Thomas Zaslavsky, \emph{Inside-out polytopes}, Advances in
  Mathematics \textbf{205} (2006), no.~1, 134--162.

\bibitem[CGGT14]{CGGT}
Nidia~Y Caicedo, Carlos~A G{\'o}mez, Jhonny~C G{\'o}mez, and Carlos~A Trujillo,
  \emph{$ b\_ $\{$h$\}$[g] $ modular sets from $ b\_ $\{$h$\}$ $ modular sets},
  arXiv preprint arXiv:1411.5741, 2014.

\bibitem[DJKL{\etalchar{+}}16]{DKLRS}
Domingos Dellamonica~Jr, Yoshiharu Kohayakawa, Sang~June Lee, Vojt{\v e}ch
  R{\"o}dl, and Wojciech Samotij, \emph{The number of b3-sets of a given
  cardinality}, Journal of Combinatorial Theory, Series A \textbf{142} (2016),
  44--76.

\bibitem[ET41]{ET}
Paul Erdos and P{\'a}l Tur{\'a}n, \emph{On a problem of sidon in additive
  number theory, and on some related problems}, J. London Math. Soc \textbf{16}
  (1941), no.~4, 212--215.

\bibitem[Gre01]{Green}
Ben Green, \emph{The number of squares and $ b\_h [g] $ sets}, Acta Arithmetica
  \textbf{100} (2001), 365--390.

\bibitem[JTT21]{JTT}
Griffin Johnston, Michael Tait, and Craig Timmons, \emph{Upper and lower bounds
  on the size of $ b\_k [g] $ sets}, arXiv preprint arXiv:2105.03706, 2021.

\bibitem[O'B04]{O'Bryant}
Kevin O'Bryant, \emph{A complete annotated bibliography of work related to
  sidon sequences}, The Electronic Journal of Combinatorics \textbf{1000}
  (2004), DS11--Jul.

\bibitem[Sid32]{Sidon}
Simon Sidon, \emph{Ein satz {\"u}ber trigonometrische polynome und seine
  anwendung in der theorie der fourier-reihen}, Mathematische Annalen
  \textbf{106} (1932), no.~1, 536--539.

\bibitem[Sin38]{Singer}
James Singer, \emph{A theorem in finite projective geometry and some
  applications to number theory}, Transactions of the American Mathematical
  Society \textbf{43} (1938), no.~3, 377--385.

\end{thebibliography}
\bibliographystyle{amsalpha}

\end{document}